\documentclass[11pt,oneside,leqno]{amsart}
\usepackage{amsxtra}
\usepackage{amsopn}
\usepackage{amsmath,amsthm,amssymb}
\usepackage{amscd}
\usepackage{amsfonts}
\usepackage{latexsym}
\usepackage{verbatim}
\usepackage{pb-diagram}

\theoremstyle{plain}
\newtheorem{theorem}{Theorem}[section]
\newtheorem*{theorem*}{Theorem}
\newtheorem{definition}[theorem]{Definition}

\newtheorem{prop}[theorem]{Proposition}
\newtheorem{cor}[theorem]{Corollary}
\newtheorem{rem}[theorem]{Remark}
\newtheorem{ex}[theorem]{Example}
\newtheorem*{mt*}{Main Theorem}
\newcommand\g{{\mathfrak{g}}}
\newcommand\ka{{\mathfrak{k}}}
\newcommand\z{{\mathfrak{z}}}

\renewcommand\O{{\omega}}
\renewcommand\t{{\theta}}

\newcommand\R{{\mathbb R}}
\newcommand\Z{{\mathbb Z}}

\newcommand\ad{{\rm ad}}

\begin{document}
\title{A class of Sasakian $5$-manifolds}
\author{Adri\'an Andrada, Anna Fino and Luigi Vezzoni}
\date{\today}

\address{Adri\'an Andrada: CIEM-FaMAF, Universidad
Nacional de C\'ordoba\\5000 C\'ordoba, Argentina}
\email{andrada@mate.uncor.edu}

\address{Anna Fino, Luigi Vezzoni: Dipartimento di Matematica \\ Universit\`a di Torino\\
Via Carlo Alberto 10 \\
10123 Torino\\ Italy} \email{annamaria.fino@unito.it, luigi.vezzoni@unito.it}
\thanks{This work was partially supported by  Conicet, ANPCyT, Secyt-UNC (Argentina), MIUR,  GNSAGA (Italy) and  by the WWS Project  \lq \lq Geometria Riemanniana e Gruppi di Lie"}

\begin{abstract} 
We obtain some general results on Sasakian Lie algebras and prove as a consequence that a $(2n + 1)$-dimensional nilpotent Lie group admitting left-invariant Sasakian structures is isomorphic to the real Heisenberg group $H_{2n + 1}$. Furthermore, we classify Sasakian Lie algebras of dimension $5$ and determine which of them carry a Sasakian $\alpha$-Einstein structure. We show that a $5$-dimensional solvable Lie group with a left-invariant Sasakian structure and which admits a compact quotient by a discrete subgroup is isomorphic to either $H_5$ or a semidirect product $\R \ltimes (H_3 \times \R)$. In particular, the compact quotient is an $S^1$-bundle over a $4$-dimensional K\"ahler solvmanifold.
\end{abstract}

\maketitle

\section{Introduction}

A Sasakian structure is the analogous in odd dimensions of a K\"ahler structure. Indeed, by \cite{BG}  a Riemannian manifold $(M,g)$  of odd dimension $2n + 1$ admits a compatible Sasakian structure if and only if the  Riemannian cone $M\times \R^+$  is K\"ahler.

In dimension 3 a  homogeneous Sasakian manifold  has to be a Lie group endowed with a left-invariant Sasakian structure by \cite{Perrone}. Therefore the classification of $3$-dimensional Sasakian homogeneous spaces depends on the classification of $3$-dimensional Sasakian Lie algebras.

By \cite{PV} a compact, simply connected, $5$-dimensional homogeneous contact manifold is diffeomorphic to the $5$-dimensional sphere $S^5$ or to the product of two spheres $S^2 \times S^3$. Moreover, both $S^5$ and $S^2\times S^3$ carry Sasakian-Einstein structures (see \cite{BGK,OP}).
Other explicit examples of Sasakian-Einstein $5$-manifolds have been found in \cite{GMSW}, while toric Sasakian manifolds in dimension 5 have been 
studied in \cite{BGO} and \cite{Co}. A classification of Sasakian-Einstein $5$-manifolds of 
cohomogeneity $1$ has been obtained in \cite{Conti} and a
classification of $5$-dimensional  Lie groups  endowed with a left-invariant contact  structure was  obtained in \cite{Diatta}.

As far as we know in the literature the only result about $5$-dimensional  Lie groups admitting left-invariant  Sasakian structures is in the  case of  nilpotent Lie groups. Indeed, in \cite{Ugarte} it was shown that the only $5$-dimensional nilpotent Sasakian Lie algebra is the real Heisenberg Lie algebra $\mathfrak h_5$. It was proved in \cite{BGO} that the real Heisenberg Lie group $H_{2n + 1}$ admits no bi-invariant Sasakian structure, even if it has a left-invariant Sasakian structure as well as a right-invariant one. These two structures are different.

The aim of this paper is to classify $5$-dimensional Lie groups endowed with a left-invariant Sasakian structure. This is equivalent to determining all $5$-dimensional Sasakian Lie algebras. The  study of Sasakian Lie algebras  is done taking into account the center of the Lie algebra, which can only be trivial or $1$-dimensional. In this way we obtain some general results on Sasakian Lie algebras in any dimensions. In particular, we show that the only $(2n + 1)$-dimensional nilpotent Lie algebra admitting a Sasakian structure is the real Heisenberg Lie algebra ${\mathfrak h}_{2n + 1}$.


By using the previous general results and independently  from  the list obtained in \cite{Diatta},  we obtain a classification up to isomorphism of $5$-dimensional Sasakian Lie algebras.

\begin{mt*}\label{main}
Let $\g$ be a $5$-dimensional Lie algebra admitting a Sasakian structure. Then
\begin{itemize}
\item if $\g$ has non-trivial center $\mathfrak{z}(\g)$, then $\g$ is solvable with  $\dim \mathfrak{z}(\g)=1$ and the quotient $\g/\mathfrak{z}(\g)$ carries an induced K\"ahler structure $($see  Theorem $\ref{classwithcenter})$;
\vspace{0.1cm}
\item if $\g$ has trivial center, then it is isomorphic to one of the following Lie algebras:
the direct products $\mathfrak{sl}(2,\R)\times\mathfrak{aff}(\R)$, $\mathfrak{su}(2)\times\mathfrak{aff}(\R)$, or  the non-unimodular  solvable Lie algebra $\g_0 \cong \R^2 \ltimes {\mathfrak h}_3$ with structure equations
\begin{equation}\label{g0}
\begin{array}{l}
[e_1,e_3]=e_3\,,\quad [e_1,e_4]=\frac12 e_4\,,\quad [e_1,e_5]=\frac12 e_5,\\
{}[e_2, e_4]=e_5 \,,\quad [e_2,e_5]=-e_4, \quad \, [e_4, e_5] = - e_3,
 \end{array}
\end{equation}
where $\mathfrak{aff}(\R)$ is the Lie algebra of the Lie group of affine motions of $\R$ and  ${\mathfrak h}_3$ is the real $3$-dimensional Heisenberg Lie algebra $($see Theorem $\ref{classtrivialcenter})$.
\end{itemize}
\end{mt*}

As a consequence we obtain   that $\g$ is  either solvable or  a direct product  of  a  $3$-dimensional  semisimple ideal with the radical $\mathfrak{aff}(\R)$.

In the case of non-trivial center we determine the list of the
$5$-dimensional Sasakian Lie algebras by using the classification of $4$-dimensional K\"ahler Lie algebras given by Ovando in \cite{Ovando}.

Moreover, we prove that the only $5$-dimensional simply connected Lie groups with a left-invariant Sasakian structure which admit a compact quotient by a discrete subgroup are the real Heisenberg group $H_5$ or a semidirect product $\R \ltimes (\R \times H_3)$ (see Corollary \ref{compactquotient}). By \cite{Hasegawa} a solvmanifold, i.e. a compact quotient of a solvable Lie group by a discrete subgroup, endowed with a K\"ahler structure is a finite quotient of a complex torus. We show that a compact quotient of a  $5$-dimensional solvable Lie group with a left-invariant Sasakian structure by a uniform discrete subgroup  is an $S^1$-bundle over a $4$-dimensional K\"ahler solvmanifold.

By \cite[Proposition 4.2]{AFFU}  $\g_0$ is the only solvable (non-nilpotent) $5$-dimensional Lie algebra admitting a Sasakian $\alpha$-Einstein structure. We show that a $5$-dimensional Sasakian $\alpha$-Einstein Lie algebra is isomorphic either to ${\mathfrak h}_5, {\mathfrak g}_0$ or to ${\mathfrak {sl}} (2, \R) \times {\mathfrak {aff}} (\R)$.

Moreover, by \cite{Diatta2} it is known that a Lie algebra of dimension at least $5$ cannot admit a Sasakian-Einstein structure.

\smallskip

\noindent{\em{Acknowledgements}} We would like to thank Liviu Ornea for useful comments and the referees for helpful suggestions and remarks.

\section{Preliminaries}

A triple $(\Phi, \alpha,  \xi )$ on a $(2n+1)$-dimensional manifold $M$  is  an
{\em almost contact structure} if $\xi$ is a nowhere vanishing vector field,  $\alpha$ is a
$1$-form, and $\Phi$  is a tensor of type $(1, 1)$ such that
\begin{equation}\label{almostcontact}
\alpha(\xi) = 1,   \quad \Phi^2 = - {\rm I} + \xi \otimes \alpha.
\end{equation}
The vector field  $\xi$  defines the characteristic foliation $\mathcal F$  with  1-dimensional leaves,
and the kernel of  $\alpha$  defines the codimension one sub-bundle ${\mathcal D} = \ker  \alpha$. Then there is
the canonical splitting of the tangent bundle $TM$ of $M$
$$
TM = {\mathcal D} \oplus {\mathcal L} ,
$$
where ${\mathcal L}$  is the trivial line bundle generated by $\xi$. Note that conditions \eqref{almostcontact} imply
\begin{equation} \label{alphaphi}
\Phi(\xi)=0\,,\quad \alpha\circ\Phi=0\,.
\end{equation}

If the $1$-form  $\alpha$ satisfies the condition
$$
\alpha \wedge ({\rm d} \alpha)^n \neq 0,
$$
then the subbundle $\mathcal D$ defines a \emph{contact structure} on $M$. In this case the vector field  $\xi$
is called the {\em Reeb vector field} and $\alpha$ is called a \emph{contact form}. Contact structures can be considered as the odd-dimensional
counterpart of symplectic structures. If $\alpha$ is a contact form, then the associated Reeb vector field satisfies
\begin{equation} \label{dalpha}
{\rm d}\alpha(\xi,X)=0
\end{equation}
for any vector field $X$ on $M$.

Similarly to the case of an almost complex structure, there is the notion of integrability of an almost contact structure. Indeed, an almost contact structure $(\Phi, \alpha,  \xi )$ is called \emph{normal} if the Nijenhuis tensor $N_{\Phi}$ associated to the tensor $\Phi$ defined by
\begin{equation}\label{NPhi}
N_{\Phi} (X, Y) = {\Phi}^2 [X, Y] + [\Phi X, \Phi Y] - \Phi [ \Phi X, Y] - \Phi [X, \Phi Y],
\end{equation}
satisfies the condition
$$
N_{\Phi} =-{\rm d}\alpha\otimes \xi.
$$
This last condition is equivalent to requiring that the almost complex structure
\begin{equation} \label{complexproduct}
J \left( X, f  \frac {\partial} {\partial t} \right) = \left(\Phi X - f \xi,\alpha (X)  \frac {\partial} {\partial t} \right)
\end{equation}
on the product $M \times \R$ be integrable,  where $f$ is a smooth function on $ M \times \R$ and $t$ is the coordinate on $\R$ (see \cite{HS}).

A Riemannian metric $g$ on an almost contact manifold $(M, \Phi, \alpha,  \xi )$  is {\em compatible} with the almost contact structure if
$$
g (\Phi X, \Phi Y) = g(X, Y) - \alpha (X) \alpha(Y),
$$
for any vector fields $X, Y$. In this case the structure $(\Phi, \alpha,  \xi, g)$  is called an {\em almost contact metric} structure. Any almost contact structure admits a compatible metric.

An almost contact metric structure $(\Phi, \alpha,  \xi, g)$  is said to be {\em contact metric} if
$$
2g (X, \Phi Y) = {\rm d} \alpha (X, Y)\,.
$$
In this case $\alpha$ is a contact form and we denote
$$\omega (X, Y) = g (X, \Phi Y).$$

\smallskip

\begin{definition} \cite{Sasaki}
A Sasakian structure  is a normal contact metric  structure, i.e.  an almost contact metric structure  $(\Phi, \alpha, \xi, g)$ such that
$$
N_ {\Phi} = -{\rm d} \alpha \otimes \xi, \quad {\rm d} \alpha = 2 \omega.
$$
\end{definition}

A Sasakian structure can be also characterized in terms of the Riemannian cone over the manifold. More precisely, we recall that a Riemannian manifold $(M,g)$ admits a compatible Sasakian structure if and only if the cone $M\times \R^+$ equipped with the metric $h=t^2g+{\rm d}t\otimes{\rm d}t$ is K\"ahler (see for instance \cite{BG}). Furthermore, in this case the Reeb vector field is Killing and the covariant derivative of $\Phi$ with respect to the Levi-Civita connection of $g$ is given by
\[ (\nabla_X \Phi)(Y)=g(\xi, Y)X-g(X,Y)\xi \]
for any pair of vector fields $X$ and $Y$ on $M$ (see for instance \cite{BG}).

If a $(2n+ 1)$-dimensional manifold $M$ admits a Sasakian structure, the  product  metric on $M \times \R$ is compatible with the  complex structure $J$ given by \eqref{complexproduct} and moreover by \cite[Proposition 3.5]{Vaisman} the corresponding Hermitian structure is locally conformally K\"ahler  with parallel Lee form.

A contact metric structure $(\Phi, \alpha,  \xi,  g)$ is Sasakian if and only if its Riemannian curvature tensor satisfies the condition
$$
R(X, Y) \xi = \alpha (Y) X - \alpha (X) Y,
$$
for any vector fields $X$ and $Y$ (see for instance \cite{BL}). This implies that the Ricci tensor ${\rm Ric}_ g$ associated to a Sasakian metric satisfies
$$
{\rm Ric}_ g(\xi,X)=2n\,\alpha(X)
$$
for any vector field $X$ on $M$, where ${\dim M}=2n+1$. In particular
${\rm Ric}_ g(\xi,\xi)=2n$ and the metric $g$ is never Ricci-flat.

When the Ricci curvature tensor  is given by
${\rm Ric}_g =  \lambda g +  \nu \,\alpha \otimes \alpha$, for some constants $\lambda,\, \nu \in \R$, the Sasakian  structure  is  called  $\alpha$-Einstein (\cite{Oku62, Sas65,BGM}). In his original definition Okumura assumed that both $\lambda$ and $\nu$ are functions, and then he showed,  as for the Einstein metrics, that they must be constant when  the dimension of the manifold is greater than three. When $\nu  = 0$, the Sasakian structure  is called  Sasakian-Einstein. These structures have been intensively studied  by many authors (see for instance 
\cite{BG2,BGK,GMSW,GMSY} and  references therein). Finally, we recall that a $5$-dimensional manifold is Sasakian $\alpha$-Einstein if and only if it is Sasakian-hypo (see \cite{CS}).

\

\section{Sasakian Lie algebras}

In this section we will begin our study of left-invariant  Sasakian structures on Lie groups. Such a structure corresponds to a Sasakian structure on the associated Lie algebra.

\begin{definition}
A Sasakian structure on a Lie algebra $\g$ is a quadruple $(\Phi,\alpha,\xi,g)$, where $\Phi\in {\rm End}(\g)$, $\alpha\in\g^*$, $\xi\in \g$ and $g$
is an inner product on $\g$ such that
$$
\begin{aligned}
&\alpha(\xi) = 1,   \quad \Phi^2 = - {\rm I} + \xi \otimes \alpha\,,\quad  g (\Phi X, \Phi Y) = g(X, Y) - \alpha (X) \alpha(Y)\,,\\
&2g (X, \Phi Y) = {\rm d} \alpha (X, Y)\,,\quad N_ {\Phi} = -{\rm d} \alpha \otimes \xi\,,
\end{aligned}
$$
where $N_\Phi$ is defined as in \eqref{NPhi}. A Lie algebra equipped with a Sasakian structure will be called a \emph{Sasakian Lie algebra.} The vector $\xi$ will be called the \emph{Reeb vector}.
\end{definition}

\begin{rem} {\rm We note that, in this setting, formula \eqref{dalpha} reads 
$$ \alpha(\xi,X)=0$$
for any $X\in\mathfrak{g}$}.
\end{rem}

\begin{ex}\emph{The classical example of a Sasakian Lie algebra is given by the $(2n+1)$-dimensional real Heisenberg Lie algebra $\mathfrak{h}_{2n+1}$.
We recall that $\mathfrak{h}_{2n+1}={\rm Span}\{X_1,\dots,X_n,Y_1,\dots,Y_n,Z\}$, where
$$
[X_i,Y_i]=Z\,,
$$
for $i=1,\dots, n$; in this case, a Sasakian structure is defined by
$$
\Phi(X_i)=Y_i\,,\,\,\Phi(Y_i)=-X_i\,,\,\, \Phi(Z)=0\,,\,\,i=1,\dots, n\,,
$$
the inner product $g$ is obtained by declaring the basis above orthonormal, $\xi=Z$ and $\alpha$ is the dual $1$-form of $Z$ with respect to the
metric $g$.}
\end{ex}

In general for a Lie algebra $\mathfrak g$ with a contact structure $\alpha$ we
can prove the following property for its center $\z (\mathfrak g)$.
\begin{prop}
Let $(\g,\alpha)$ be a contact Lie algebra with $\xi$ its Reeb vector and let $\z (\mathfrak g)$ be the center of $\g$. Then
\begin{enumerate}
\item[1.] $\dim \z (\mathfrak g)\leq 1\,;$
\vspace{0.1cm}
\item[2.] if $\dim \z (\mathfrak g)=1$, then $ \z(\mathfrak g)=\R\,\xi$.
\end{enumerate}
\end{prop}
\begin{proof}
The first item is well known and follows from the fact that ${\rm d}\alpha$ is non-degenerate on  $\ker\alpha$. For the second item we
fix an arbitrary generator $Z$ of $\z (\mathfrak g)$. We may write
$Z=a\,\xi+X$, where $a\in\R$ and $X\in \ker\alpha$. We have
$$
0=\alpha([Z,Y])=\alpha([a\xi+X,Y])=\alpha([X,Y])=-{\rm d}\alpha(X,Y)
$$
for all $Y\in \ker\alpha$. By the non-degeneracy of ${\rm d}\alpha$ on $\ker\alpha$, it follows that $X=0$ and, consequently, $\z (\mathfrak g)=\R\,\xi$.
\end{proof}

\begin{rem} {\rm We recall that due to \cite[Theorem 5]{BW}, the only semisimple Lie algebras admitting a contact form are $\mathfrak{su}(2)$ and $\mathfrak{sl}(2,\R)$}.
\end{rem}

\medskip

\subsection{Non-trivial center}

We show that in the case of  Sasakian Lie algebras with non-trivial center  the kernel of the contact form
inherits a natural structure of K\"ahler Lie algebra. Moreover two Sasakian Lie algebras are isomorphic if and only if the corresponding K\"ahler Lie algebras are equivalent. This allows us to use the classification of the $4$-dimensional K\"ahler Lie algebras of \cite{Ovando} to classify $5$-dimensional Sasakian Lie algebras with non-trivial center.

We start by considering the following
\begin{prop}
Let $(\g,\Phi,\alpha,\xi,g)$ be a Sasakian Lie algebra with non-trivial center $\mathfrak{z}(\g)$ generated by $\xi$. Then
the quadruple $(\ker\alpha,\theta,  \Phi,g)$ is a K\"ahler Lie algebra, where $\t$ is the component of the Lie bracket of $\g$ on $\ker\alpha$.
\end{prop}
\begin{proof}
Let $X,Y,Z$ in $\ker\alpha$, then
$$
\begin{aligned}
0=&[X,[Y,Z]]+[Z,[X,Y]]+[Y,[Z,X]]\\
 =&[X,\t(Y,Z)+\alpha([Y,Z])\,\xi]+[Z,\t(X,Y)+\alpha([X,Y])\,\xi]\\
  &\quad+[Y,\t(Z,X)+\alpha([Z,X])\,\xi]\\
 =&[X,\t(Y,Z)]+[Z,\t(X,Y)]+[Y,\t(Z,X)]\\
 =&\t(X,\t(Y,Z))+\t(Z,\t(X,Y))+\t(Y,\t(Z,X))\\
  &+\alpha([X,\t(Y,Z)]+[Z,\t(X,Y)]+[Y,\t(Z,X)])\,\xi
\end{aligned}
$$
i.e.
$$
\begin{aligned}
\t(X,\t(Y,Z))+\t(Z,\t(X,Y))+\t(Y,\t(Z,X))=0\,,\quad {\rm d}\omega=0\,,
\end{aligned}
$$
$2 \omega$ being the restriction of ${\rm d}\alpha$ on $\ker\alpha$. Then $\theta$ defines a Lie bracket on $\ker \alpha$. Since $\omega$ is non-degenerate, the statement holds.
\end{proof}
Conversely, let $(\mathfrak{h}, [ \, , \,  ]_{\mathfrak h}, \O, \Phi,g)$ be a K\"ahler Lie algebra and set $\g=\mathfrak{h}\oplus \R\,\xi$. Then defining
$$
[X,Y]=[X,Y]_{\mathfrak{h}}-\O(X,Y)\,\xi
$$
for $X,Y \in \mathfrak{h}$ and
$$
[\xi,\mathfrak{h}]=0
$$
we obtain a new Lie algebra  $({\mathfrak g}, [ \, , \, ])$ equipped with a natural Sasakian structure, where the contact form $\alpha$ on $\g$ is defined as
$$
\alpha(a\,\xi+X)=a
$$
for all $X\in\mathfrak{h}$ and $\Phi$ and $g$ are extended in a natural way.
\begin{cor} \label{quotient}
Let $(\g,\Phi,\alpha,\xi,g)$ be a Sasakian Lie algebra with non-trivial center $\mathfrak{z}(\g)$ generated by $\xi$. Then
$\g/\mathfrak{z}(\g)$ inherits a natural K\"ahler structure.
\end{cor}

Now we have the following easy-to-prove proposition which will be used in $\S$\ref{5}.
\begin{prop} \label{equivalence}
Two Sasakian Lie algebras with non-trivial center $(\g_1,\Phi_1,\alpha_1,\xi_1,g_1)$, $(\g_2,\Phi_2,\alpha_2,\xi_2,g_2)$  are isomorphic
if and only  $\ker\alpha_1$ and $\ker\alpha_2$ are isomorphic as K\"ahler Lie algebras.
\end{prop}


Since a nilpotent Lie algebra has always non-trivial center, we can apply the results above in order to determine all  nilpotent Lie algebras admitting a Sasakian structure. It is known that in dimensions 3 and 5 the only nilpotent Sasakian Lie algebras are the real Heisenberg algebras $\mathfrak h_{3}$ and $\mathfrak h_{5}$, respectively (see \cite{Geiges} and \cite[Corollary 5.5]{Ugarte}). We show next that this still holds in any dimension.

\begin{theorem}
Let $\g$ be a $(2n+1)$-dimensional nilpotent Lie algebra admitting a Sasakian structure. Then $\g$ is isomorphic to the $(2n+1)$-dimensional Heisenberg Lie algebra.
\end{theorem}

\begin{proof}
Let $(\Phi,\alpha,\xi,g)$ be a Sasakian structure on $\g$. Since $\g$ is nilpotent it has non-trivial center
$\mathfrak{z}(\g)= \R\,\xi$. The quotient $\g/\mathfrak{z}(\g)$ is a K\"ahler nilpotent Lie algebra, hence it is unimodular, and therefore, using a result of Hano \cite{Hano}, it is  flat. As a consequence, $\g/\mathfrak{z}(\g)$ is abelian. This implies immediately that $\g$ is isomorphic to the Heisenberg Lie algebra.
\end{proof}

\begin{rem}
\emph{According to \cite{TV} the $(2n+1)$-dimensional Heisenberg Lie algebra admits a contact Calabi-Yau structure. In particular $\mathfrak{h}_{2n+1}$ admits Sasakian $\alpha$-Einstein structures.}
\end{rem}

\medskip

\subsection{Trivial center}

In the case the Sasakian Lie algebra $\mathfrak g$ has  trivial center, we have the following properties for $\ad_{\xi}$.

\begin{prop} \label{kerad}
Let $(\g,\Phi,\alpha,\xi,g)$ be a Sasakian Lie algebra. Then
\begin{enumerate}
\item[1.] $\ad_{\xi} \Phi= \Phi \,\ad_{\xi}$, and therefore $\ker\ad_\xi$ and ${\rm Im}\,\ad_\xi$ are $\Phi$-invariant subspaces of $\g$\,;
\vspace{0.1cm}
\item[2.] $\ad_\xi \Phi$ is symmetric with respect to $g\,;$
\vspace{0.1cm}
\item[3.]  $\ad_{\xi}$ is skew-symmetric with respect to $g$ and therefore $({\rm Im}\,\ad_\xi)^{\perp}=\ker \ad_\xi\,.$
\end{enumerate}
\end{prop}
\begin{proof}
The first item is an easy consequence of $N_{\Phi}=-{\rm d}\alpha\otimes \xi$ and \eqref{alphaphi}. Indeed, for any $X\in\g$,
$$
\begin{aligned}
0=N_{\Phi} (\xi,X)=&[\Phi \xi, \Phi X]-\Phi[\Phi \xi,X]-\Phi [\xi,\Phi X]+ \Phi^2[\xi,X]\\
          =&-\Phi [\xi, \Phi X]-[\xi,X]\,,
\end{aligned}
$$
i.e.
$$
[\xi, \Phi X]= \Phi [\xi,X]\,.
$$

The second item follows from the following computation, for $X,Y\in \g$,
$$
\begin{aligned}
2 g([\xi, \Phi X],Y)&=2 g( \Phi [\xi,X],Y)=-2 g([\xi,X], \Phi Y)=-{\rm d}\alpha([\xi,X],Y)\\
             &=\alpha([[\xi,X],Y])=-\alpha([[X,Y],\xi]-[[\xi,Y],X])\\
             &=\alpha([[\xi,Y],X])=-{\rm d}\alpha([\xi,Y],X)=-2 g([\xi,Y], \Phi X)\\
             &=2 g(X,[\xi, \Phi Y])\,.
             \end{aligned}
$$
Finally, let $X, Y\in \mathfrak g$. We can write $X=a\xi+ \Phi X'$, with $X' \in \ker \alpha$, and thus the third item follows from
$$
\begin{aligned}
g([\xi, X], Y) &= g ([\xi, a \xi + \Phi X'], Y) =  g( [\xi, \Phi X'], Y) = g (X', [\xi, \Phi Y])\\
 &= - g(\Phi  X', [\xi, Y]) = - g(X - a \xi, [\xi, Y]) = - g(X, [\xi, Y]).
\end{aligned}
$$
This shows  also that $\ker\ad_\xi\subseteq({\rm Im\,\ad_\xi})^{\perp}$. For dimensional reasons we have that $\ker\ad_\xi=({\rm Im\,\ad_\xi})^{\perp}$.
\end{proof}
As a direct consequence of Proposition \ref{kerad} we have the following
\begin{cor} \label{corad}
Let $(\g,\Phi,\alpha,\xi,g)$ be a Sasakian Lie algebra. Then there is an orthogonal decomposition
$$
\g=\ker \ad_\xi {\oplus}\, {\rm Im}\, \ad_\xi.
$$
\end{cor}

\begin{prop} \label{sasakisubalgebra}
Let $(\g,\Phi,\alpha,\xi,g)$ be a Sasakian Lie algebra with  trivial center. 
\begin{enumerate}
\item  If  $\dim \g\geq 5$, then  $\ker\ad_\xi$ is a Sasakian Lie subalgebra of $\g$ with non-trivial center.

\item If $X \in \ker \ad_{\xi}, Y \in {\rm Im \,\ad_\xi}$, then $[X, Y] \in {\rm Im \,\ad_\xi}$.
\end{enumerate}
\end{prop}
\begin{proof}
$(1)$ Since $\ad_\xi$ is a derivation, its kernel is a Lie subalgebra of $\g$. Furthermore, let $X\in\ker\alpha\cap\ker\ad_\xi $ and $Y\in\g$, then
$$
{\rm d}\alpha(X,\ad_\xi Y)=-\alpha([X,[\xi,Y]])=\alpha([Y,[X,\xi]])+\alpha([\xi,[Y,X]])=0\,.
$$
Hence if ${\rm d}\alpha(X,Z)=0$ for any $Z\in \ker\ad_\xi$, then ${\rm d}\alpha(X,W)=0$ for any $W\in\ker \alpha$ and consequently $X=0$.
It follows that the restriction of $\alpha$ to $\ker\ad_\xi$ is a contact form and $\xi$ is still the Reeb vector. Clearly $\xi$ belongs to the
center of this subalgebra and the restrictions of $\Phi$ and $g$ induce a Sasakian structure on
$\ker\ad_\xi$.

\noindent $(2)$ We can write
$$
X = a \xi + X', \quad Y = [\xi, Y'],
$$
with $a \in \R$, $X' \in \ker \ad_{\xi} \cap \ker \alpha$, $Y' \in \g$. Then, by using Jacobi identity, 
$$
[X, Y] = a [\xi, [\xi, Y']] + [X', [\xi, Y']] =  a [\xi, [\xi, Y']] - [\xi, [Y', X']],
$$
which belongs to  ${\rm Im \,\ad_\xi}$.
\end{proof}

\begin{rem}
\emph{Note that if
$\ker \alpha={\rm Im}\,\ad_\xi$, then the commutator ideal $\g'=[\g,\g]$ coincides with $\g$.}

\end{rem}
Assume now that the Sasakian Lie algebra $(\g,\Phi,\alpha,\xi,g)$ with trivial center satisfies in addition the condition  $\g'\neq \g$.

With respect to the decomposition $\g=\ker \ad_\xi\oplus {\rm Im}\, \ad_\xi$, we can write
$$
(\ad_{\xi})_{|\ker\alpha}=
\begin{pmatrix}
0& 0\\
0& U
\end{pmatrix}\,,\quad
\Phi_{|\ker\alpha}=\begin{pmatrix}
A& C\\
B& D
\end{pmatrix},
$$
where $U\colon {\rm Im}\,\ad_\xi\to {\rm Im}\,\ad_\xi$ is a non-singular operator. Therefore from the equality $\Phi \, \ad_\xi=\ad_\xi \Phi$ one gets that
$$
B=C=0\,,\quad DU=UD\,,
$$
and since $(\Phi_{|\ker\alpha})^2=-{\rm I}$,
$$
A^2=D^2=-{\rm I}\,.
$$
In particular, if $\g$ is solvable, then the Reeb vector  $\xi$ cannot belong to the commutator $\g'$.

\medskip

\subsection{$3$-dimensional Sasakian Lie algebras}

Simply connected homogeneous $3$-dimensional contact metric manifolds were classified by Perrone in \cite{Perrone}, showing that the homogeneous space has to be a Lie group with a left-invariant contact metric structure. Among these Lie groups we  can find the ones that admit a Sasakian structure.

For the sake of completeness we perform  the classification of  $3$-dimensional Sasakian Lie groups even if it  is already known also by  \cite{Geiges,CC}.  Indeed, Geiges in \cite{Geiges} has determined  the diffeomorphism types of compact  manifolds of dimension $3$ that admit a Sasakian  structure.

\begin{theorem} Any $3$-dimensional Sasakian Lie algebra is isomorphic to one of the following: $\mathfrak{su}(2)$, $\mathfrak{sl}(2,\R)$, ${\mathfrak {aff}} (\R) \times \R$,
$\mathfrak{h}_3$.
\end{theorem}

\begin{proof}
It is well known that, if the commutator $\g'$ of  a $3$-dimensional Lie algebra $\g$ coincides with $\g$, then  $\g$ is semisimple and it is isomorphic to either $\mathfrak{su}(2)$ or $\mathfrak{sl}(2,\R)$. Both Lie algebras admit a Sasakian structure (see for instance \cite{Geiges,CC,Perrone}).

If $\g' \neq \g$ then $\g$ is solvable and, by Corollary \ref{corad}, we have the orthogonal decomposition
$$
\g=\ker \ad_\xi{\oplus}\, {\rm Im}\, \ad_\xi,
$$
where $\xi$ is the Reeb vector. Since $ \ker \ad_\xi  \cap \ker \alpha$ is $\Phi$-invariant, the only  possibility is that $\dim (\ker \ad_\xi  \cap \ker \alpha) =2$ and thus $\xi$ belongs to the center $\z (\g)$.
The K\"ahler quotient $\g/\z( \g)$ is then isomorphic to $\R^2$ or ${\mathfrak {aff}} (\R)$. It is easy to show that in the former case $\g$ is isomorphic to ${\mathfrak h}_3$.  In the latter case there exists a basis  $\{ e^1, e^2, e^3 \}$
of $\g^*$ such that
$$
{\rm d} e^1 =0, \quad {\rm d} e^2 = e^{12},  \quad {\rm d} e^3 = 2 e^{12},
$$
with  respect to which the Sasakian structure is
$$
\xi = e_3,  \quad  \alpha = e^3,  \quad \Phi (e_1) = e_2,  \quad \omega =  e^{12}.
$$
Considering the new basis
$$
E^j = e^j, \, j = 1,2, \quad E^3 = e^3 - 2 e^2,
$$
we have the isomorphism of $\g$  with ${\mathfrak {aff}} (\R) \times \R$.
\end{proof}

\

\section{$5$-dimensional Sasakian Lie algebras}\label{5}

In order to classify  $5$-dimensional Sasakian Lie algebras $\mathfrak g$ up to isomorphism we will consider separately the case of Lie algebras with trivial or non-trivial center.

\subsection{$5$-dimensional Sasakian Lie algebras with non-trivial center}
Proposition \ref{equivalence} together with  the classification in  \cite{Ovando} of $4$-dimensional
K\"ahler Lie algebras allow us to classify $5$-dimensional Sasakian Lie algebras with
non-trivial center.
\begin{theorem} \label{classwithcenter}
Any $5$-dimensional Sasakian Lie algebra  $\g$ with non-trivial center is isomorphic to one of the following Lie algebras
$$
\begin{aligned}
&\g_1=\left(0,0,0,0,e^{12}+e^{34}\right)\simeq \mathfrak{h}_5\,;\\
&\g_2=\left(0,-e^{12},0,0,e^{12}+e^{34}\right)\simeq \mathfrak{aff}(\R)\times \mathfrak{h}_3\,;\\
&\g_3=\left(0,-e^{13},e^{12},0,e^{14}+e^{23}\right)\simeq \R\ltimes (\mathfrak{h}_3\times \R)\,;\\
&\g_4=\left(0,-e^{12},0,-e^{34},e^{12}+e^{34}\right)\simeq \mathfrak{aff}(\R)\times\mathfrak{aff}(\R)\times\R\,;\\
&\g_5=\left(\frac12\,e^{14},\frac12\,e^{24},-e^{12}+e^{34},0,e^{12}-e^{34}\right)\simeq \R\times(\R\ltimes \mathfrak{h}_3)\,;\\
&\g_6=\left(2e^{14},-e^{24},-e^{12}+e^{34},0,e^{23}\right)\simeq\R\ltimes \mathfrak{n}_4\,\,;\\
&\g_7^{\delta}=\left(\frac{\delta}{2}e^{14}+e^{24},-e^{14}+\frac{\delta}{2}e^{24},-e^{12}+\delta e^{34},0,e^{12}-\delta e^{34} \right)\simeq \R\times(\R\ltimes \mathfrak{h}_3)\,,\,\,
\delta>0\,;\\
&\g_8^{\delta}=\left(e^{14},\delta\,e^{34},-\delta\,e^{24},0,e^{14}+e^{23}\right)\simeq \R\ltimes_\delta (\mathfrak{h}_3\times \R)\,,\,\,\delta>0\,.\\
\end{aligned}
$$
\end{theorem}
\begin{proof}
By Corollary \ref{quotient},  if  $(\g,\Phi,\alpha, \xi,g)$ is  a $5$-dimensional Sasakian Lie algebra, then the quotient $\mathfrak g/ \z(\mathfrak g)$ is a $4$-dimensional K\"ahler Lie algebra with the K\"ahler structure induced by the Sasakian one. We may choose a basis $\{ e_1, \ldots, e_5\}$ of $\mathfrak g$ such that
$$
\xi = e_5,  \quad \alpha = e^5,  \quad \mathfrak g/ \z(\mathfrak g) = {\rm Span}  \{ e_1, e_2, e_3, e_4 \}
$$
and ${\rm d} e^5 =  2 \Omega$, where $\Omega$ is the K\"ahler form on the quotient.

$4$-dimensional K\"ahler Lie algebras  and their  possible K\"ahler forms have been classified in \cite{Ovando}.
By using this classification we obtain  that $\mathfrak g$ is isomorphic to one of the following  Lie algebras
$$
\begin{aligned}
&\ka_1=\left(0,0,0,0,\lambda\,e^{12}+\mu\,e^{34}\right),\quad \lambda\,,\mu<0\,;\\
&\ka_2=\left(0,-e^{12},0,0,\lambda\,e^{12}+\mu\,e^{34}\right),\quad \lambda\,,\mu<0\,;\\
&\ka_3=\left(0,-e^{13},e^{12},0,\lambda\,e^{14}+\mu\,e^{23}\right)\,,\quad \lambda\,,\mu<0\,;\\
&\ka_4=\left(0,-e^{12},0,-e^{34},\lambda\,e^{12}+\mu\,e^{34}\right)\,,\quad \lambda\,,\mu<0\,;\\
&\ka_5=\left(\frac12\,e^{14},\frac12\,e^{24},-e^{12}+e^{34},0,\lambda(e^{12}-e^{34})\right)\,,\quad \lambda<0\,;\\
&\ka_6=\left(2e^{14},-e^{24},-e^{12}+e^{34},0,\lambda\,e^{14}+\mu e^{23}\right)\,,\quad \lambda\,,\mu<0\,;\\
&\ka_7^-=\left(\frac{\delta}{2}e^{14}+e^{24},-e^{14}+\frac{\delta}{2}e^{24},-e^{12}+\delta e^{34},0,\lambda( e^{12}-\delta e^{34}) \right)\,,\quad
\delta>0\,,\lambda<0\,;\\
&\ka_7^+=\left(\frac{\delta}{2}e^{14}+e^{24},-e^{14}+\frac{\delta}{2}e^{24},-e^{12}+\delta e^{34},0,\lambda( e^{12}-\delta e^{34}) \right)\,,\quad
\delta>0\,;\\
&\ka_8^-=\left(e^{14},\delta\,e^{34},-\delta\,e^{24},0,\lambda\,e^{14}+\mu\,e^{23}\right)\,,\quad\delta\,,\lambda>0\,,\mu<0;\\
&\ka_8^+=\left(e^{14},\delta\,e^{34},-\delta\,e^{24},0,\lambda\,e^{14}+\mu\,e^{23}\right)\,,\quad\delta\,,\lambda\,,\mu>0\,;\\
\end{aligned}
$$
equipped with the Sasakian structure $(\Phi,\alpha=e^5,\xi=e_5,g)$, where  $\Phi$ is  given  respectively by
$$
\begin{array} {lll}
& \Phi_1(e_1)=e_2\,,\quad &\Phi_1(e_3)=e_4\\
& \Phi_2(e_1)=e_2\,,\quad &\Phi_2(e_3)=e_4\\
& \Phi_3(e_1)=e_4\,,\quad & \Phi_3(e_2)=e_3\\
&  \Phi_4(e_1)=e_2\,,\quad &  \Phi_4(e_3)=e_4\\
&  \Phi_{5}(e_1)=e_2\,,\quad &  \Phi_{5}(e_4)=e_3\\
&  \Phi_6(e_2)=e_3\,,\quad &  \Phi_6(e_4)=-2\,e_1\\
&  \Phi_{7}^-(e_1)=e_2\,,\quad &  \Phi_{7}^-(e_4)=e_3\\
&  \Phi_{7}^+(e_1)=-e_2\,,\quad &  \Phi_{7}^+(e_4)=-e_3\\
&  \Phi_8^-(e_4)=e_1\,,\quad &  \Phi_8^-(e_2)=e_3\\
&  \Phi_8^+(e_4)=e_1\,,\quad &  \Phi_8^+(e_3)=e_2\\
\end{array}
$$
and $2 g  ( \cdot  \, ,  \cdot \; ) =  {\rm d} e^5 (\Phi \, \cdot  , \,  \cdot )$.

By a direct computation one may check the following isomorphisms
$$
\ka_i \cong \g_i,  \, i = 1, \ldots, 5, \quad \ka_6^{\pm} \cong \g_6^{\tau} \,  ({\rm with} \, \tau ={ \lambda} /{\mu}), \quad \ka_j^{\pm} \cong \g_j^{\delta}, \, j = 7, 8.
$$

For the family $$\g_6^{\tau}=\left(2e^{14},-e^{24},-e^{12}+e^{34},0,\tau \,e^{14}+e^{23}\right),\,\tau>0$$ we can show that every  element in the family   is isomorphic to $\g_6$  by considering the new basis
$$
E_1 =  2 e_1 + \tau e_5, \quad E_j = e_j, j = 2, \ldots, 5.
$$

By Proposition \ref{equivalence}, the Lie algebras $\mathfrak g_i$ and $\mathfrak g_j^{\delta}$ are  not isomorphic for any $i$ and $j$. Moreover,   $\mathfrak g_i$ (respectively  $\mathfrak g_i^{\delta})$   is not isomorphic to   $\mathfrak g_k$ (respectively  $\mathfrak g_k^{\delta})$ for any $i \neq k$.

Applying again Proposition \ref{equivalence}, it follows that the Lie algebras  in the family $\g_7^{\delta}$  are not isomorphic one each other for different values of $\delta$. The same holds for the family $\g_8^{\delta}$.

\end{proof}

As a consequence  we have the following
\begin{cor} \label{compactquotient}
A unimodular Sasakian Lie algebra with non-trivial center is isomorphic either to the nilpotent
Heisenberg Lie algebra $\mathfrak h_5$ or the solvable Lie algebra $\g_3$. The simply connected Lie group $G_3$ with Lie algebra $\g_3$ admits a co-compact discrete subgroup $\Gamma$.
\end{cor}

\begin{proof}
By a direct computation one may check that the only unimodular  Sasakian Lie algebras with non-trivial center are in  fact ${\mathfrak g}_1 \cong  {\mathfrak h}_5$ and ${\mathfrak g}_3 \cong \R \ltimes ( {\mathfrak h}_3 	\times \R)$. The  solvable simply connected Lie group $G_3$ is isomorphic to $\R^5$ with the following product
{\small $$
 \begin{pmatrix} x_1\\ x_2\\ x_3\\ x_4\\ x_5  \end{pmatrix}
 \cdot   \begin{pmatrix} y_1\\ y_2\\ y_3\\ y_4\\ y_5  \end{pmatrix} =
  \begin{pmatrix} x_1 + y_1\\  \cos (x_1) y_2 + \sin(x_1) y_3 + x_2\\ - \sin(x_1) y_2 + \cos(x_1) y_3 + x_3\\ Q \\ x_5 + y_5 \end{pmatrix},
$$}
where \[ Q=y_4 +   x_1 y_5 - \sin^2(x_1) y_2 y_3   - \frac 14 \sin (2x_1)(y_2^2 - y_3^2) +x_2\left(-\sin(x_1)y_2+\cos(x_1)y_3\right) +x_4.\]

The discrete subgroup
$$
\Gamma = \left \{ \left(2 \pi m_1, m_2, m_3, m_4, \frac  {1} {2 \pi} m_5 \right) \, \vert  \, m_i \in \Z  \right \}
$$
acts freely and properly discontinuously on $G_3$. Moreover, the quotient manifold $\Gamma \backslash G_3$ is compact.
\end{proof}

The solvmanifold $\Gamma \backslash G_3$ is by construction the total space of an $S^1$-bundle over a $4$-dimensional non completely solvable K\"ahler solvmanifold. This class of K\"ahler solvmanifolds was found by Hasegawa in \cite{Hase3} (see also \cite{Hase2}); all of these solvmanifolds are hyperelliptic surfaces. We recall that a solvmanifold is called completely solvable if the adjoint representation of the corresponding solvable Lie group has only real eigenvalues.

\medskip

\subsection{$5$-dimensional Sasakian Lie algebras with trivial center}\label{sectrivialcenter}

Let $(\g,\Phi,\alpha,\xi,g)$ be a $5$-dimensional Sasakian Lie algebra with trivial center and $\mathfrak g'  = \mathfrak g$. By \cite[Section 5] {Diatta} the only contact Lie algebra with the above property is the semidirect product $\mathfrak {sl} (2, \R) \ltimes \R^2$, where $\mathfrak {sl} (2, \R)$ acts on $\R^2$ by matrix multiplication. We can prove the following

\begin{prop}
The Lie algebra $\mathfrak {sl} (2, \R) \ltimes \R^2$ does not admit any Sasakian structure.
\end{prop}

\begin{proof}
Let $\{ e_1, e_2, e_3 \}$ be the standard basis of $\mathfrak {sl}(2, \R)$ with Lie brackets
$$
[e_1, e_2] = 2 e_2, \quad [e_1, e_3] = - 2 e_3, \quad [e_2, e_3] =e_1
$$
and let $\{ e_4, e_5 \}$ be the canonical  basis of $\R^2$. Then $\mathfrak {sl} (2, \R) \ltimes \R^2$ has structure equations
$$
\begin{array}{l}
{\rm d} e^1 = - e^{23},\\
{\rm d} e^2 =  - 2 e^{12},\\
{\rm d} e^3 = 2 e^{13},\\
{\rm d} e^4 = - e^{14} - e^{25},\\
{\rm d} e^5 = e^{15} - e^{34}.
\end{array}
$$
A $1$-form $\alpha=\sum_{i = 1}^5  a_i e^i$ is contact if and only if the real numbers $a_i$ satisfy the condition
\begin{equation} \label{Delta}
 \Delta:= a_3 a_4^2 -  a_2 a_5^2 - a_1 a_4 a_5 \neq 0.
\end{equation}
In this case, the corresponding Reeb vector is given by
$$
\xi = - \frac {1} {3 \Delta}  \left( a_4 a_5 e_1 + a_5^2 e_2 - a_4^2 e_3 + (a_1 a_5 - 2 a_3 a_4) e_4 + (a_1 a_4 + 2 a_2 a_5) e_5 \right).
$$
One can check that $X\in\mathfrak{sl}(2,\R)\ltimes \R^2$ belongs to $\ker \ad_{\xi} \cap \ker \alpha $ if and only  $$X = t (- a_5 e_4 + a_4 e_5), \quad t \in \R.
$$
If there exists a Sasakian structure with contact form $\alpha$, then $\ker \ad_{\xi} \cap \ker \alpha $ has to be $\Phi$-invariant (see Proposition \ref{kerad}) and this is only possible if $a_4 = a_5 =0$ in contrast with \eqref{Delta}.

\end{proof}

Now we can consider the case  of  $5$-dimensional Sasakian Lie algebras with trivial center and such that $\g'\neq \g$. In this case
$$
\dim \ker (\ad_\xi)_{|\ker\alpha}=\dim {\rm Im}( \ad_\xi)=2\,.
$$
It is easy to see that there exists an orthonormal basis $\{e_1,\dots,e_4\}$ of $\ker\alpha$ with respect to which $\Phi_{|\ker\alpha}$ can be written as
$$
\Phi_{|\ker \alpha} =
\begin{pmatrix}
 0& 1 & 0&0\\
 -1& 0& 0&0\\
 0& 0& 0& 1\\
 0& 0& -1 &0
\end{pmatrix}
$$
and $\ker\ad_{\xi}=\text{Span}\{\xi,e_1,e_2\}$, ${\rm Im}\,\ad_{\xi}=\text{Span}\{e_3,e_4\}$. Moreover in this basis
$$
(\ad_\xi)_{|\ker\alpha}=
\begin{pmatrix}
 0& 0 & 0&0\\
 0& 0& 0&0\\
 0& 0& a& -b\\
 0& 0& b &a
\end{pmatrix}\,.
$$
Note that in terms of $\{e_1,\dots,e_4\}$ the  $2$-form ${\rm d}\alpha$ takes the standard form ${\rm d}\alpha=2\,(e^{12}+e^{34})$. Furthermore, taking into account that $\alpha([e_3,e_4])=-{\rm d}\alpha(e_3,e_4)=-2$, and recalling that $\theta:\g\times\g\to\ker\alpha$ denotes the projection of the bracket on $\g$ onto $\ker\alpha$, we have
$$
\begin{aligned}
0=&[\xi,[e_3,e_4]]+[e_4,[\xi,e_3]]-[e_3,[\xi,e_4]]\\
 =&[\xi,\t(e_3,e_4)-2\,\xi]+[e_4,a\,e_3+b\,e_4]-[e_3,-b\,e_3+a\,e_4]\\
 =&[\xi,\t(e_3,e_4)]+a\,[e_4,e_3]-a\,[e_3,e_4]\\
 =&[\xi,\t(e_3,e_4)]-2a\,[e_3,e_4]\\
 =&[\xi,\t(e_3,e_4)]-2a\,\t(e_3,e_4)+4a\,\xi
\end{aligned}
$$
which implies
\begin{equation}\label{a=0}
a=0\,,\quad \t(e_3,e_4)\in\ker \ad_\xi\cap\ker\alpha\,.
\end{equation}
From now on we set $e_5=\xi$ and denote by
$\{e^{1},\dots,e^{5}\}$ the dual basis of $\{e_1,\dots,e_5\}$.
Since $b\neq 0$, by a suitable rescaling of $\alpha$ and $\xi$ we may assume $b=\pm1$.
We will examine separately the two cases: Case A: $b= 1$ and Case B:  $b = - 1$.

\smallskip
We start by considering the case  A.
Taking into account \eqref{a=0}  and Proposition \ref{sasakisubalgebra}, we can write
\begin{equation} \label{structureeqA}
\begin{aligned}
&{\rm d}e^1=a_1\,e^{12}+a_6\,e^{34}\,,\\
&{\rm d}e^2=b_1e^{12}+b_6\,e^{34}\,,\\
&{\rm d}e^3=-e^{45}+c_2\,e^{13}+c_3\,e^{14}+c_4\,e^{23}+c_5\,e^{24}\,,\\
&{\rm d}e^4=e^{35}+f_2\,e^{13}+f_3\,e^{14}+f_4\,e^{23}+f_5\,e^{24}\,,\\
&{\rm d}e^5=2(e^{12}+e^{34})\,.
\end{aligned}
\end{equation}
From the condition $N_{\Phi}=-{\rm d} e^5 \otimes e_5$
we get
$$
c_5=c_2- f_3 - f_4, \quad f_5 =f_2  + c_3+ c_4.
$$
Now we impose ${\rm d}^2=0$.  From the vanishing of the coefficients of $e^{ij5}$, $i, j = 1, \ldots, 4$, in ${\rm d}^2 e^k =0, k = 1, \ldots, 5$, we get the following linear equations
\begin{equation} \label{d2}
 c_2 - f_3 = f_2 + c_3 =0.
\end{equation}
Moreover,
$$
{\rm d}^2 e^5 =  (a_6   + 2 c_4 + f_2 + c_3) e^{234} + (- b_6 + c_2 + f_3) e^{134}.
$$
and therefore in addition to  \eqref{d2} we have
$$
 a_6 = -2 c_4, \quad  b_6 =2 c_2.
 $$
In this way the structure equations \eqref{structureeqA}
 of $\mathfrak g$ reduce to
$$
\begin{array} {l}
{\rm d} e^1 = a_1 e^{12} -2 c_4 e^{34},\\
{\rm d}  e^2 = b_1 e^{12}  + 2 c_2 e^{34},\\
{\rm d} e^3 = -e^{45} +c_2 e^{13} + c_3 e^{14}  + c_4 e^{23} - f_4 e^{24},\\
{\rm d} e^4 = e^{35} -c_3 e^{13} +c_2 e^{14} +  f_4 e^{23} + c_4 e^{24},\\
{\rm d} e^5 = 2 (e^{12} + e^{34})
\end{array}
$$
with the structure constants satisfying the conditions
$$
\begin{array}{l}
c_2 (a_1  +2  c_4) =0,  \quad c_4 (a_1 +2 c_4) =0, \quad  c_2 (- b_1 +2 c_2)=0,\\
 c_4 ( - b_1 + 2 c_2)=0, \quad
 a_1 c_3 -b_1 f_4 + 2 =0, \quad c_2 a_1+c_4 b_1=0.
\end{array}
$$
We get the following solutions for the above system
$$
\begin{aligned}
{\rm A}1)\,\,  & b_1 = c_2 = c_4 =0,  c_3 \neq 0, a_1  = - \frac {2} {c_3},\\
{\rm A}2)\,\,  & c_2 = c_4 =0,  b_1 \neq 0, f_4 = \frac{2 + a_1 c_3} {b_1} ,\\
{\rm A}3)\,\,  & b_1=c_2 =0, a_1 \neq 0, c_3  = - \frac{2} {a_1}, c_4 = - \frac{1}{2} a_1,\\
{\rm A}4)\,\,  &  b_1 \neq 0, c_2 = \frac 12 b_1, c_4 = - \frac 12 a_1, f_4 =  \frac{2 + a_1 c_3} {b_1}.
\end{aligned}
$$

In the first two cases, we have that
$$
{\rm A}1), \, {\rm A} 2)\simeq \mathfrak{aff}(\R)\times \mathfrak{sl}(2,\R)
$$
where respectively
$$
{\rm A}1) \begin{cases}
 \begin{aligned}
&\mathfrak{aff}(\R)\simeq{\rm Span}\{f_4\,e_1-c_3\,e_2,e_1-c_3\,e_5\}\,,\\
&\mathfrak{sl}(2,\R)\simeq{\rm Span}\{e_3,e_4,e_5\}\,.
\end{aligned}
\end{cases}
$$
and
$$
{\rm A}2) \begin{cases}
 \begin{aligned}
&\mathfrak{aff}(\R)\simeq{\rm Span}\{a_1\,e_1+ b_1\,e_2 + 2 e_5,e_1-c_3\,e_5\}\,,\\
&\mathfrak{sl}(2,\R)\simeq{\rm Span}\{e_3,e_4,e_5\}\,.
\end{aligned}
\end{cases}
$$

In the other cases we see that
$$
{\rm A}3),  {\rm A}4)  \cong \R^2 \ltimes {\mathfrak h}_3
$$
by using for A3)  the new basis
$$
\left \{ E_1 = a_1 e_1 + 2 e_5, E_2 = \frac{1}{a_1} e_2, E_j =  e_j, j =3,4,5\right \},
$$
with $\R^2 =  {\rm {Span}} \{ E_2, E_5 \}$, $ \mathfrak h_3 = {\rm {Span}} \{ E_1, E_3, E_4 \}$ and
$$
\ad_{E_2} = \begin{pmatrix} 1&0&0\\ 0&\frac 12&\frac{f_4}{a_1}\\ 0&-\frac{f_4}{a_1}&\frac12 \end{pmatrix}, \quad \ad_{E_5} = \begin{pmatrix} 0&0&0\\ 0&0&-1\\ 0&1&0 \end{pmatrix}.
$$
For  A4) we may choose the new basis
$$
\left \{ F_1 = a_1 e_1 + b_1 e_2 +  2 e_5, F_2 =\frac{1}{b_1} e_1, F_j =  e_j, j =3,4,5\right \},
$$
with $\R^2 =  {\rm {Span}} \{ F_2, F_5 \}$, $ \mathfrak h_3 = {\rm {Span}} \{ F_1, F_3, F_4 \}$ and
$$
\ad_{F_2} = \begin{pmatrix} -1&0&0\\ 0&-\frac12&-\frac{c_3}{b_1}\\ 0&\frac{c_3}{b_1}&-\frac12 \end{pmatrix}, \quad \ad_{F_5} = \begin{pmatrix} 0&0&0\\ 0&0&-1\\ 0&1&0 \end{pmatrix}.
$$

Now we study the case B. Taking into account \eqref{a=0} and  Proposition \ref{sasakisubalgebra}, we can write
\begin{equation} \label{structureeqB}
\begin{aligned}
&{\rm d}e^1=a_1\,e^{12}+a_6\,e^{34}\,,\\
&{\rm d}e^2=b_1e^{12}+b_6\,e^{34}\,,\\
&{\rm d}e^3=e^{45}+c_2\,e^{13}+c_3\,e^{14}+c_4\,e^{23}+c_5\,e^{24}\,,\\
&{\rm d}e^4=-e^{35}+f_2\,e^{13}+f_3\,e^{14}+f_4\,e^{23}+f_5\,e^{24}\,,\\
&{\rm d}e^5=2(e^{12}+e^{34})\,.
\end{aligned}
\end{equation}
Condition $N_{\Phi}=-{\rm d} e^5 \otimes e_5$ implies the following linear equations
$$
c_5=c_2-f_3-f_4\,,\,f_5=f_2+c_3+c_4\,,
$$
while ${\rm d}^2=0$ gives
$$
c_2=f_3\,,\,
f_2=-c_3\,,\,a_6=-2c_4\,,\,f_3=\frac12b_6\,.
$$
Hence the structure equations \eqref{structureeqB} of $\g$ reduces to
$$
\begin{aligned}
&{\rm d}e^1=a_1\,e^{12}-2c_4\,e^{34}\,,\\
&{\rm d}e^2=b_1\,e^{12}+b_6\,e^{34}\,,\\
&{\rm d}e^3=e^{45}+\frac12 b_6\,e^{13}+c_3\,e^{14}+c_4\,e^{23}-f_4\,e^{24}\,,\\
&{\rm d}e^4=-e^{35}-c_3\,e^{13}+\frac12 b_6\,e^{14}+f_4\,e^{23}+c_4\,e^{24}\,,\\
&{\rm d}e^5=2(e^{12}+e^{34})\,,
\end{aligned}
$$
where the structure constants are related by the equations
$$
\begin{aligned}
&c_4 (2c_4+a_1)=0\,,\,b_6 (a_1+2c_4)=0\,,\, c_4 (b_1- b_6)=0\,,\,b_6(b_6 -b_1)=0\\
&b_1c_4+\frac12 a_1b_6=0\,,\,c_3 a_1- b_1 f_4-2=0\,.
\end{aligned}
$$
This imposes to consider the following four cases:
$$
\begin{aligned}
{\rm B}1)\,\,&b_1=0\,,\, b_6=0\,,\,c_4=0\,,\,c_3=\frac{2}{a_1},\\
{\rm B}2)\,\, &b_6=0\,,\, c_4=0\,,\, f_4=\frac{c_3a_1-2}{b_1},\\
{\rm B}3)\,\,&b_1= 0\,,\,b_6=0\,,c_3=\frac{2}{a_1}\,,\,c_4= -\frac12 a_1,\\
{\rm B}4)\,\,& b_6=b_1\,,\,c_4=-\frac12 a_1\,,\, f_4= \frac{c_3a_1-2}{b_1}\,.
\end{aligned}
$$
In the first two cases we have
$$
{\rm B}1)\,,\,{\rm B}2)\simeq \mathfrak{aff}(\R)\times \mathfrak{su}(2)\,,
$$
where respectively
$$
{\rm B}1) \begin{cases}
 \begin{aligned}
&\mathfrak{aff}(\R)\simeq{\rm Span}\{a_1\,e_1+ e_5\,,f_4\,e_1+e_5\}\\
&\mathfrak{su}(2)\simeq{\rm Span}\{e_3,e_4,e_5\}
\end{aligned}
\end{cases}
$$
and
$$
{\rm B}2) \begin{cases}
 \begin{aligned}
&\mathfrak{aff}(\R)\simeq{\rm Span}\{a_1\,e_1+ b_1\,e_2 + 2 e_5,e_1+{c_3}\,e_5\}\\
&\mathfrak{su}(2)\simeq{\rm Span}\{e_3,e_4,e_5\} \,.
\end{aligned}
\end{cases}
$$
Again in the cases
${\rm B}3)$ and ${\rm B}4)$ $\g$ is solvable  and
$$
{\rm B}3), {\rm B}4)  \cong \R^2 \ltimes {\mathfrak h}_3\,
$$
by using for B3)  the new basis
$$
\left \{ G_1 = a_1 e_1 + 2 e_5, G_2 = \frac{1}{a_1} e_2, G_j =  e_j, j =3,4,5\right \},
$$
with $\R^2 =  {\rm {Span}} \{ G_2, G_5 \}$, $ \mathfrak h_3 = {\rm {Span}} \{ G_1, G_3, G_4 \}$ and
$$
\ad_{G_2} = \begin{pmatrix} 1&0&0\\ 0&\frac 12&\frac{f_4}{a_1}\\ 0&-\frac{f_4}{a_1}&\frac12 \end{pmatrix}, \quad \ad_{G_5} = \begin{pmatrix} 0&0&0\\ 0&0&1\\ 0&-1&0 \end{pmatrix}.
$$
For B4) we may choose the new basis
$$\left \{ H_1 = a_1 e_1 + b_1 e_2 +  2 e_5, H_2 =\frac{1}{b_1} e_1, H_j =  e_j, j =3,4,5\right\},
$$
with $\R^2 =  {\rm {Span}} \{ H_2, H_5 \}$, $ \mathfrak h_3 = {\rm {Span}} \{H_1, H_3, H_4 \}$ and
$$
\ad_{H_2} = \begin{pmatrix} -1&0&0\\ 0&-\frac12&-\frac{c_3}{b_1}\\ 0&\frac{c_3}{b_1}&-\frac12 \end{pmatrix}, \quad \ad_{H_5} = \begin{pmatrix} 0&0&0\\ 0&0&1\\ 0&-1&0 \end{pmatrix}.
$$

In the cases  ${\rm A}3)\,,\, \rm{A}4)\,,\, \rm{B}3)\,,\, \rm{B}4)$, the corresponding Lie algebras are all isomorphic to a semidirect product
$\mathfrak{g}_t=\R^2{\ltimes}_{\psi_t}\mathfrak{h}_3$, where $\R^2={\rm Span}\{X,Y\}$,  $\mathfrak{h_3}={\rm Span}\{v_1,v_2,v_3\}$ and
$$
[v_2,v_3]=-v_1\,,\quad \psi_t(X)= \begin{pmatrix} 1&0&0\\ 0&\frac12&t\\ 0&-t&\frac12 \end{pmatrix}, \quad\psi_t(Y)= \begin{pmatrix} 0&0&0\\ 0&0&-1\\ 0&1&0 \end{pmatrix}.
$$
By changing $X$ to $X+t\, Y$ we obtain an isomorphism between $\mathfrak{g}_t$ and $\g_0$ with structure equations \eqref{g0}.

To summarize we can state the following

\begin{theorem} \label{classtrivialcenter}
If  a $5$-dimensional Sasakian Lie algebra $\g$ has trivial center, then it is isomorphic to one of the following Lie algebras: the direct products $\mathfrak{sl}(2,\R)\times\mathfrak{aff}(\R)$, $\mathfrak{su}(2)\times\mathfrak{aff}(\R)$, or  the non-unimodular solvable Lie algebra $\g_0$.
\end{theorem}

\begin{rem}\emph{
Note that $\g_0$ corresponds to the Lie algebra numbered 22 in the classification of $5$-dimensional solvable contact Lie algebras provided by Diatta in \cite{Diatta}.  }
\end{rem}

\

\section{$5$-dimensional Sasakian $\alpha$-Einstein Lie algebras}

In this section we study Sasakian $\alpha$-Einstein Lie algebras. A Sasakian Lie algebra $(\g,\Phi,\alpha,\xi,g)$ is called $\alpha$-Einstein if the Ricci tensor ${\rm Ric}_g$  of the metric $g$ satisfies ${\rm Ric}_g =  \lambda g +  \nu \,\alpha \otimes \alpha$ for some $\lambda, \,\nu\in\R$.

It is known that the canonical Sasakian structure on $\mathfrak{h}_5$ is $\alpha$-Einstein.  Furthermore, in view of \cite[Proposition 4.2]{AFFU} the Lie algebra $\g_0$ from Theorem \ref{classtrivialcenter} is the only solvable (non nilpotent) $5$-dimensional Lie algebra admitting a Sasakian $\alpha$-Einstein structure. Thus, in order to determine all the $5$-dimensional Lie algebras admitting such a structure, we only have to consider the non-solvable ones, which are $\mathfrak{sl}(2,\R)\times\mathfrak{aff}(\R)$ and $\mathfrak{su}(2)\times\mathfrak{aff}(\R)$ according to Theorems \ref{classwithcenter} and \ref{classtrivialcenter}.

\begin{prop}
The Lie algebra $\mathfrak{sl}(2,\R)\times\mathfrak{aff}(\R)$ admits Sasakian $\alpha$-Einstein structures, while there are none Sasakian $\alpha$-Einstein structures on $\mathfrak{su}(2)\times\mathfrak{aff}(\R)$.
\end{prop}

\begin{proof}
All Sasakian structures on $\mathfrak{sl}(2,\R)\times\mathfrak{aff}(\R)$ are described in the cases A1), A2) given in \S\ref{sectrivialcenter}.
Performing standard computations we obtain that in the case A1) the Ricci tensor is given by the following matrix
$$
{\rm Ric}_g=\begin{pmatrix}
-\frac{2c_3^2+2}{c_3^2} &0 &0  &0 &0\\
0 &-\frac{2c_3^2+2}{c_3^2} &0  &0 &0\\
0 &0 &-4  &0 &0\\
0 &0 &0  &-4 &0\\
0 &0 &0  &0 &4\\
\end{pmatrix}\,.
$$
Therefore for $c_3=\pm 1$ we see that the Sasakian structure is $\alpha$-Einstein. In the case A2) we have
$$
{\rm Ric}_g=\begin{pmatrix}
-(2+a_1^2+b_1^2) &0 &0  &0 &0\\
0 &-(2+a_1^2+b_1^2) &0  &0 &0\\
0 &0 &-4  &0 &0\\
0 &0 &0  &-4 &0\\
0 &0 &0  &0 &4\\
\end{pmatrix}\,.
$$
Therefore for $a_1^2+b_1^2=2$ we obtain that the Sasakian structure is $\alpha$-Einstein.

We consider now the Lie algebra $\mathfrak{su}(2)\times\mathfrak{aff}(\R)$. For this algebra the Sasakian structures
are described by the cases B1),
B2) studied in \S\ref{sectrivialcenter}. In the case B1) the Ricci tensor is given by
$$
{\rm Ric}_g=\begin{pmatrix}
-(2+a_1^2) &0 &0  &0 &0\\
0 &-(2+a_1^2) &0  &0 &0\\
0 &0 &0  &0 &0\\
0 &0 &0  &0 &0\\
0 &0 &0  &0 &4\\
\end{pmatrix}\,,
$$
whereas in the case B2) it is given by
$$
{\rm Ric}_g=\begin{pmatrix}
-(2+a_1^2+b_1^2) &0 &0  &0 &0\\
0 &-(2+a_1^2+b_1^2) &0  &0 &0\\
0 &0 &0  &0 &0\\
0 &0 &0  &0 &0\\
0 &0 &0  &0 &4\\
\end{pmatrix}\,.
$$
As a consequence the Sasakian structures on this Lie algebra never satisfy the $\alpha$-Einstein condition.
\end{proof}

To sum up, we can now state the following
\begin{theorem}
The only $5$-dimensional Lie algebras admitting a Sasakian $\alpha$-Einstein structure are $\mathfrak{h}_5$, $\g_0$ and $\mathfrak{sl}(2,\R)\times\mathfrak{aff}(\R)$.
\end{theorem}


\begin{thebibliography}{12}

\bibitem{AFFU} de Andr\'es L. C., Fern\'andez M., Fino A., Ugarte L., Contact 5-manifolds with SU(2)-structure, preprint arXiv:0706.0386, to appear in \emph {Q. J. Math.}

\bibitem{BL}
Blair D.E.:\emph{ Riemannian Geometry of Contact and Symplectic Manifolds}, Birkh\"auser (2002).

\bibitem{Boeckx} Boeckx, E.,
A full classification of contact metric $(\kappa,\mu)$-spaces, \emph{Illinois J. Math.} {\bf 44} (2000),  212--219.

\bibitem{BW}
Boothby W. M., Wang H. C.: On contact manifolds, \emph{Ann. of Math.} (2)  {\bf 68} (1958), 721--734.

\bibitem{BG}
Boyer C. P., Galicki K.: 3-Sasakian manifolds, \emph{Surveys in differential geometry: essays on Einstein manifolds} 123--184, Surv. Differ. Geom., VI, Int. Press, Boston, MA, 1999.

\bibitem{BG2} Boyer C. P., Galicki, K.: On Sasakian-Einstein geometry, \emph {Internat. J. Math.} {\bf 11} (2000), 873--909.

\bibitem{BGK}
Boyer C. P., Galicki K., Koll\'ar J.: Einstein metrics on spheres,  \emph{Ann. of Math.} (2)  {\bf 162}  (2005),  no. 1, 557--580.

\bibitem{BGM} Boyer C.P, Galicki K., Matzeu P.:  On Eta-Einstein Sasakian Geometry, \emph {Commun. Math. Phys.} {\bf  262} (2006), 177--208.

\bibitem{BGO}
Boyer C. P., Galicki K., Ornea L.: Constructions in Sasakian geometry,  \emph{Math. Z.}  {\bf 257}  (2007),  no. 4, 907--924.

\bibitem{CC} Cho, J. T.,  Chun  S.  H.,  On the classification of contact Riemannian manifolds satisfying the condition (C), \emph{Glasg. Math. J.}  {\bf 45} (2003), 475--492.

\bibitem{Co} van Coevering C.: Toric Surfaces and Sasakian-Einstein $5$-manifolds,  preprint arXiv:math/0607721, PhD dissertation.

\bibitem{Conti}
Conti D.: Cohomogeneity one Einstein-Sasaki 5-manifolds, \emph{Comm. Math. Phys.}  {\bf 274}  (2007),  no. 3, 751--774.

\bibitem{CS}
Conti D., Salamon S.: Generalized Killing spinors in dimension 5, \emph{Trans. Amer. Math. Soc.} {\bf 359} (2007), 5319--5343.

\bibitem{Diatta}
Diatta A.: Left-invariant contact structures on Lie groups, \emph{ Diff. Geom. Appl.} {\bf 26} (2008), no. 5, 544--552.

\bibitem{Diatta2}
Diatta A.: Riemannian geometry on contact Lie groups, \emph{ Geom.  Dedicata}  {\bf 133}  (2008), 83--94.

\bibitem{GMSW}
Gauntlett J. P., Martelli D., Sparks J., Waldram D.: Sasaki-Einstein metrics on $S\sp 2\times S\sp 3$, \emph{Adv. Theor. Math. Phys.}  {\bf 8}  (2004),  no. 4, 711--734.

\bibitem{GMSY} Gauntlett J.  P.,  Martelli D.  Sparks J., Yau, S.-T.:  Obstructions to the existence of Sasaki-Einstein metrics, \emph{Comm. Math. Phys.} {\bf 273} (2007),  803--827.

\bibitem{Geiges} Geiges  H., Normal contact structures on $3$-manifolds, \emph{T\^{o}hoku Math. J.}  (2) {\bf 49} (1997), 415--422.

\bibitem{Hase3} Hasegawa, K.: A class of compact K\"ahlerian solvmanifolds and a general conjecture, \emph{ Geom.  Dedicata} {\bf 78} (1999), 253--258.

\bibitem{Hase2} Hasegawa, K.: Complex and K\"ahler structures on compact solvmanifolds, Conference on Symplectic Topology, \emph{J. Symplectic Geom.} {\bf 3} (2005), no. 4, 749--767. 

\bibitem{Hasegawa} Hasegawa K.: A note on compact solvmanifolds with K\"ahler structures, \emph{Osaka J. Math.} {\bf 43} (2006), 131--135.

\bibitem{Hano} Hano J.: On K\"ahlerian homogeneous spaces of unimodular Lie groups, \emph {Amer. J. Math.} {\bf 79} (1957), 885--900.

\bibitem{Oku62} Okumura M.: Some remarks on space with a certain contact structure, \emph{T\^ohoku Math. J.}
(2) (1962) {\bf 14}, 135--145.

\bibitem{OP} Ornea L., Piccinni P.: Induced Hopf bundles and Einstein metrics, in \emph{New developments in differential geometry}, Budapest (1996), 295-306, Kluwer.

\bibitem{Ovando}
Ovando G.: Invariant pseudo-K\"ahler metrics in dimension four, \emph{J. Lie Theory}  {\bf 16} (2006),  371--391.

\bibitem{PV} Perrone D., Vanhecke L.: Five-dimensional homogeneous contact manifolds and related problems,  \emph{T\^ohoku Math. J.} {\bf 43} (1991),  243--248.

\bibitem{Perrone} Perrone D.: Homogeneous contact Riemannian three-manifolds,  \emph{Illinois J. Math.} {\bf 42} (1998), 243--256.

\bibitem{Sasaki} Sasaki S.: On differentiable manifolds with certain structures which are closely related to almost contact structure, \emph {T\^{o}hoku Math. J.} {\bf 2} (1960), 459--476.

\bibitem{Sas65} Sasaki, S.: Almost Contact Manifolds, Part 1. Lecture Notes, Mathematical Institute,
T\^ohoku University, 1965.

\bibitem{HS} Sasaki S., Hatakeyama Y.: On differentiable manifolds with certain structures which are closely related to almost contact structure II, \emph{T\^{o}hoku Math. J.} (2) {\bf 13} (1961), 281--294.

\bibitem{TV}
Tomassini A., Vezzoni L.: Contact Calabi-Yau manifolds and Special Legendrian submanifolds, \emph{ Osaka J. Math.} {\bf 45} (2008), 127--147.

\bibitem{Ugarte} Ugarte L.: Hermitian structures on six-dimensional nilmanifolds, \emph {Transform. Groups} {\bf 12} (2007),  175--202.

\bibitem{Vaisman} Vaisman I.: Locally conformal K\"ahler manifolds with parallel Lee form, \emph{Rend. Mat.} (6) {\bf 12} (1979), 263--284.

\end{thebibliography}
\end{document}